\documentclass{article}

\usepackage[centertags]{amsmath}
\usepackage{hyperref}
\usepackage{amsfonts}
\usepackage{amssymb}
\usepackage{amsthm}
\usepackage{newlfont}
\usepackage{amscd}
\usepackage{amsmath,amscd}
\usepackage{graphicx}
\usepackage[all]{xy}
\usepackage{verbatim}

\usepackage{color}

\newcommand{\adj}{\rightleftarrows}
\newcommand{\ten}{\otimes}

\newcommand{\cB}{\mathcal{B}}
\newcommand{\cC}{\mathcal{C}}
\newcommand{\cD}{\mathcal{D}}

\newcommand{\cF}{\mathcal{F}}

\newcommand{\cM}{\mathcal{M}}

\newcommand{\cS}{\mathcal{S}}

\newcommand{\cW}{\mathcal{W}}

\newtheorem{thm}{Theorem}[section]

\newtheorem{cor}[thm]{Corollary}

\newtheorem{lem}[thm]{Lemma}
\newtheorem{prop}[thm]{Proposition}

\theoremstyle{definition}
\newtheorem{define}[thm]{Definition}

\theoremstyle{remark}
\newtheorem{rem}[thm]{Remark}

\def\R{\textup{R}}

\DeclareMathOperator{\Pro}{Pro}
\DeclareMathOperator{\Ind}{Ind}
\DeclareMathOperator{\Map}{Map}

\DeclareMathOperator{\Hom}{Hom}

\DeclareMathOperator{\precolim}{colim}

\DeclareMathOperator{\Lw}{Lw}

\DeclareMathOperator{\coSp}{coSp}
\def\colim{\mathop{\precolim}}

\def \mcal{\mathcal}

\DeclareFontEncoding{OT2}{}{} 
\DeclareTextFontCommand{\textcyr}{\fontencoding{OT2}\fontfamily{wncyr}\fontseries{m}\fontshape{n}\selectfont}

\begin{document}
\title{From weak cofibration categories to model categories}

\author{
Ilan Barnea 
\and
Tomer M. Schlank 
}

\maketitle

\begin{abstract}
In \cite{BaSc2} the authors introduced a much weaker homotopical structure than a model category, called a ``weak cofibration category". We further showed that a small weak cofibration category induces in a natural way a model category structure on its ind-category, provided the ind-category satisfies a certain two out of three property. The purpose of this paper is to serve as a companion to the papers above, proving results which say that if a certain property or structure exists in the weak cofibration category, then the same property or structure also holds in the induced model structure on the ind-category. Namely, we consider the property of being left proper and the structures of a monoidal category and a category tensored over a monoidal category (in a way that is compatible with the weak cofibration structure). For the purpose of future reference, we consider the more general situation where we only have an ``almost model structure" on the ind-category.
\end{abstract}

\tableofcontents

\section{Introduction}
In \cite{BaSc1} the authors introduced the concept of a weak fibration category. This is a category $\cC$, equipped with two subcategories of weak equivalence and fibrations, satisfying certain axioms. A weak fibration category is a much weaker notion than a model category and its axioms are much easily verified.

If $\cC$ is any category, its pro-category $\Pro(\cC)$ is the category of inverse systems in $\cC$. That is, objects in $\Pro(\cC)$ are diagrams $I\to\cC$, with $I$ a cofiltered category. If $X$ and $Y$ are objects in $\Pro(\cC)$ having the same indexing category, then a natural transformation $X\to Y$ defines a morphism in $\Pro(\cC)$, but morphisms in $\Pro(\cC)$ are generally more flexible.

Given a weak fibration category $\cC$, there is a very natural way to induce a notion of weak equivalence on the pro-category $\Pro(\cC)$. Namely, we define the weak equivalences in $\Pro(\cC)$ to be the smallest class of maps that contains all (natural transformations that are) levelwise weak equivalences, and is closed under isomorphisms of maps in $\Pro(\cC)$. If $\cW$ is the class of weak equivalences in $\cC$, then we denote the class of weak equivalences in $\Pro(\cC)$ by $\Lw^{\cong}(\cW)$. Note, however, that $\Lw^{\cong}(\cW)$ may not satisfy the two out of three property. Weak fibration categories for which $\Lw^{\cong}(\cW)$ satisfies the two out of three property are called pro-admissible.

The main result in \cite{BaSc1} is that a pro-admissible weak fibration category $\cC$ induces in a natural way a model structure on $\Pro(\cC)$, provided $\cC$ has colimits and satisfies a technical condition called \emph{homotopically small}. In \cite{BaSc2}, we explain that an easy consequence of this result is that any small pro-admissible weak fibration category $\cC$ induces a model structure on $\Pro(\cC)$.

Dually, one can define the notion of a weak \emph{cofibration} category (see Definition \ref{d:weak_fib}), and deduce that a small \emph{ind}-admissible weak cofibration category induces a model structure on its \emph{ind}-category (which is the dual notion of a pro-category, see Definition \ref{d:pro category}). This is the setting in which we work with in this paper, however, everything we do throughout the paper is completely dualizable, so it can also be written in the ``pro" picture.

The purpose of this paper is to continue and complete the above mentioned work. Namely, model categories can possess further properties or structures that are useful in different situations. For example, a model category can be left or right proper, simplicial, monoidal etc. Many of these properties and structures have rather straightforward analogues also in the world of weak cofibration categories. Thus, given a small ind-admissible weak cofibration category that possess such a property or structure, a natural question to ask is whether the induced model structure on its ind-category also possesses the same property or structure. In this paper we give an affirmative answer to this question for three basic properties and structures of model categories. Namely, we define the following notions:

\begin{enumerate}
\item A left proper weak cofibration category (see Definition \ref{d:r_proper}).
\item A monoidal weak cofibration category (see Definition \ref{d:monoidal}).
\item A weak cofibration category that is tensored over a monoidal weak cofibration category (see Definition \ref{d:tensored}).
\end{enumerate}
We show that if a small ind-admissible weak cofibration category possesses one of the notions above, then the induced model structure on its ind-category possesses the corresponding notion, as a model category. For a left proper weak cofibration category this is shown in Corollary \ref{c:r_proper}, for a monoidal weak cofibration category this is shown in Proposition \ref{p:monoidal}, and for a tensored weak cofibration category this is shown in Proposition \ref{p:tensored}. A special case of 3. above is the notion of a simplicial weak cofibration category (see Definition \ref{d:simplicial}). Thus we obtain that the ind-category of a small ind-admissible simplicial weak cofibration category is naturally a simplicial model category (see Proposition \ref{p:simplicial}).

In \cite{BaSc2} the notion of an \emph{almost model category} was introduced. It was used as an auxiliary notion that was useful in showing that certain weak cofibration categories are ind-admissible.
An almost model category is a quadruple $(\cM,\cW,\cF,\cC)$ satisfying all the axioms of a model category, except (maybe) the two out of three property for $\cW$. A weak cofibration category
$(\cC,\cW,\cC of)$ is called \emph{almost ind-admissible} if the class $\Lw^{\cong}(\cW)$, of morphisms in $\Ind(\cC)$, satisfies the following portion of the two out of three property:

For every pair $X\xrightarrow{f} Z\xrightarrow{g} Y $ of composable morphisms in $\Ind(\cC)$ we have:
\begin{enumerate}
\item If $f,g$ belong to $\Lw^{\cong}(\cW)$ then $g\circ f\in \Lw^{\cong}(\cW)$.
\item If $f,g\circ f$ belong to $\Lw^{\cong}(\cW)$ then $g\in \Lw^{\cong}(\cW)$.
\end{enumerate}

In \cite{BaSc2}, we show that any small almost ind-admissible weak cofibration category induces an almost model structure on its ind-category. Thus, all the results of this paper can be easily \textbf{formulated} also in this more general context. Since all the these generalized results indeed hold, with exactly the same proofs, and since future applications to appear require it, we chose to write this paper in this more general setting.

\subsection{Organization of the paper}

We begin in Section \ref{s:prelim} with a brief account of the necessary background on ind-categories and homotopy theory in ind-categories. In Section \ref{s:r_proper} we define the notion of a left proper almost model category and a left proper weak cofibration category. We then show that a small left proper almost ind-admissible weak cofibration category gives rise to a left proper almost model structure on its ind-category. In Section \ref{s:tensored} we discuss tensored and monoidal structures in ind-categories, and how they are induced from similar structures on the original categories.
In Section \ref{s:almost} we define the notions of tensored and monoidal almost model categories and tensored and monoidal weak cofibration categories. We show that such a structure on an almost ind-admissible weak cofibration category induces the corresponding structure for the almost model structure on its ind-category.


\section{Preliminaries: homotopy theory in ind-categories}\label{s:prelim}
In this section we review the necessary background on ind-categories and homotopy theory in ind-categories. We state the results without proof, for later reference. Most of the references that we quote are written for pro-categories, but we bring them here translated to the ``ind" picture for the convenience of the reader. Standard references on pro-categories include \cite{AM} and \cite{SGA4-I}. For the homotopical parts the reader is referred to \cite{BaSc}, \cite{BaSc1}, \cite{BaSc2}, \cite{EH} and \cite{Isa}.

\subsection{Ind-categories}
In this subsection we bring general background on ind-categories.

\begin{define}\label{d:cofiltered}
A category $I$ is called \emph{filtered} if the following conditions are satisfied:
\begin{enumerate}
\item The category $I$ is non-empty.
\item For every pair of objects $s,t \in I$, there exists an object $u\in I$, together with
morphisms $s\to u$ and $t\to u$.
\item For every pair of morphisms $f,g:s\to t$ in $I$, there exists a morphism $h:t\to u$ in $I$ such that $h\circ f=h\circ g$.
\end{enumerate}
\end{define}

A category is called \emph{small} if it has only a set of objects and a set of morphisms.

\begin{define}\label{d:pro category}
Let $\mcal{C}$ be a category. The category $\Ind(\mcal{C})$ has as objects all diagrams in $\cC$ of the form $I\to \cC$ such that $I$ is small and filtered. The morphisms are defined by the formula
$$\Hom_{\Ind(\mcal{C})}(X,Y):=\lim_s \colim_t \Hom_{\mcal{C}}(X_s,Y_t).$$
Composition of morphisms is defined in the obvious way.
\end{define}

Thus, if $X:I\to \mcal{C}$ and $Y:J\to \mcal{C}$ are objects in $\Ind(\mcal{C})$, providing a morphism $X\to Y$ means specifying for every $s$ in $I$ an object $t$ in $J$ and a morphism $X_s\to Y_t$ in $\mcal{C}$. These morphisms should satisfy a compatibility condition. In particular, if $p:I\to J$ is a functor, and $\phi:X\to Y\circ p=p^*Y$ is a natural transformation, then the pair $(p,\phi)$ determines a morphism $\nu_{p,\phi}:X\to Y$ in $\Ind(\cC)$ (for every $s$ in $I$ we take the morphism $\phi_s:X_{s}\to Y_{p(s)}$). Taking $X=p^*Y$ and $\phi$ to be the identity natural transformation, we see that any $p:I\to J$ determines a morphism $\nu_{p,Y}:p^*Y\to Y$ in $\Ind(\cC)$.

The word ind-object refers to objects of ind-categories. A \emph{simple} ind-object
is one indexed by the category with one object and one (identity) map. Note that for any category $\mcal{C}$, $\Ind(\mcal{C})$ contains $\mcal{C}$ as the full subcategory spanned by the simple objects.

\begin{define}\label{d:cofinal}
Let $p:I\to J$ be a functor between small categories. The functor $p$ is said to be \emph{(right) cofinal} if for every $j$ in $J$ the over category ${p}_{j/}$ is nonempty and connected.
\end{define}

Cofinal functors play an important role in the theory of ind-categories mainly because of the following well-known lemma:

\begin{lem}\label{l:cofinal}
Let $p:I\to J$ be a cofinal functor between small filtered categories, and let $X:J\to \cC$ be an object in $\Ind(\cC)$. Then the morphism in $\Ind(\cC)$ that $p$ induces, $\nu_{p,X}:p^*X\to X$, is an isomorphism.
\end{lem}

\begin{define}
Let $T$ be a poset. Then we view $T$ as a category which has a single morphism $u\to v$ iff $u\leq v$. Note that this convention is the opposite of that used in \cite{BaSc1}.
\end{define}

Thus, a poset $T$ is filtered (see Definition \ref{d:cofiltered}) iff $T$ is non-empty, and for every $a,b$ in $T$ there exists an element $c$ in $T$ such that $c\geq a,b$. A filtered poeset will also be called \emph{directed}.

\begin{define}\label{def CDS}
A cofinite poset is a poset $T$ such that for every element $x$ in $T$ the set $T_x:=\{z\in T| z \leq x\}$ is finite.
\end{define}

\begin{define}\label{def natural}
Let $\mcal{C}$ be a category with finite colimits, $M$ a class of morphisms in $\mcal{C}$, $I$ a small category, and $F:X\to Y$ a morphism in $\mcal{C}^I$. Then:
\begin{enumerate}
\item The map $F$ will be called a \emph{level-wise $M$-map}, if for every $i\in I$ the morphism $X_i\to Y_i$ is in $M$. We will denote this by $F\in \Lw(M)$.
\item The map $F$ will be called a \emph{cospecial} $M$-\emph{map}, if $I$ is a cofinite poset and for every $t\in I$ the natural map
$$X_t\coprod_{\colim_{s<t} X_s} \colim_{s<t} Y_s \to Y_t  $$
is in $M$. We will denote this by $F\in \coSp(M)$.
\end{enumerate}
\end{define}

\begin{define}\label{def mor}
Let $\mcal{C}$ be a category with finite colimits and let $M$ be a class of morphisms in $\mcal{C}$.
\begin{enumerate}
\item We denote by $\R(M)$ the class of morphisms in $\mcal{C}$ that are retracts of morphisms in $M$. Note that $\R(\R(M))=\R(M)$.
\item We denote by $M^{\perp}$ (resp. ${}^{\perp}M$) the class of morphisms in $\mcal{C}$ having the right (resp. left) lifting property with respect to all the morphisms in $M$.
\item We denote by $\Lw^{\cong}(M)$ the class of morphisms in $\Ind(\mcal{C})$ that are \textbf{isomorphic} to a morphism that comes from a natural transformation which is a levelwise $M$-map.
\item We denote by $\coSp^{\cong}(M)$ the class of morphisms in $\Ind(\mcal{C})$ that are \textbf{isomorphic} to a morphism that comes from a natural transformation which is a cospecial $M$-map.
\end{enumerate}
\end{define}

\begin{prop}[{\cite[Proposition 2.2]{Isa}}]\label{l:ret_lw}
Let $\mcal{C}$ be a category and let $M$ be a class of morphisms in $\mcal{C}$. Then $$\R(\Lw^{\cong}(M)) = \Lw^{\cong}(M).$$
\end{prop}

\begin{prop}[{\cite[Corollary 2.20]{BaSc1}}]\label{p:forF_sp_is_lw}
Let $\mcal{C}$ be a category with finite colimits, and $\mcal{M} \subseteq \mcal{C}$ a subcategory that is closed under cobase change, and contains all the isomorphisms. Then $\coSp^{\cong}(\mcal{M})\subseteq \Lw^{\cong}(\mcal{M})$.
\end{prop}

\subsection{From a weak cofibration category to an almost model category}

In this subsection we recall from \cite{BaSc2} the notion of an almost model category and some of its properties. We then discuss the construction of almost model structures on ind-categories.

\begin{define}
An almost model category is a quadruple $(\cM,\cW,\cF,\cC)$ satisfying all the axioms of a model category, except (maybe) the two out of three property for $\cW$. More precisely, an almost model category satisfies:
\begin{enumerate}
\item $\cM$ is complete and cocomplete.
\item $\cW$ is a class of morphisms in $\cM$ that is closed under retracts.
\item $\cF,\cC$ are subcategories of $\cM$ that are closed under retracts.
\item $\cC\cap \cW\subseteq{}^{\perp}\cF$  and $\cC\subseteq{}^{\perp}(\cF\cap\cW)$.
\item There exist functorial factorizations in $\cM$ into a map in $\cC\cap \cW$ followed by a map in $\cF$ and into a map in $\cC$ followed by a map in $\cF\cap \cW$.
\end{enumerate}
\end{define}

\begin{lem}[{\cite[Lemma 3.10]{BaSc2}}]\label{l:lifting}
In an almost model category $(\cM,\cW,\cF,\cC)$ we have:
\begin{enumerate}
\item $\cC\cap \cW={}^{\perp}\cF$.
\item $\cC={}^{\perp}(\cF\cap\cW)$.
\item $\cF\cap \cW=\cC^{\perp}$.
\item $\cF=(\cC\cap\cW)^{\perp}$.
\end{enumerate}
\end{lem}

\begin{define}\label{d:weak_fib}
A \emph{weak cofibration category} is a category ${\cC}$ with an additional
structure of two subcategories
$${\cC of}, {\cW} \subseteq {\cC}$$
that contain all the isomorphisms such that the following conditions are satisfied:
\begin{enumerate}
\item ${\cC}$ has all finite limits.
\item ${\cW}$ has the two out of three property.
\item The subcategories ${\cC of}$ and ${\cC of}\cap {\cW}$ are closed under cobase change.
\item Every map $A\to B $ in ${\cC}$ can be factored as $A\xrightarrow{f} C\xrightarrow{g} B $,
where $f$ is in ${\cC of}$ and $g$ is in ${\cW}$.
\end{enumerate}
The maps in ${\cC of}$ are called \emph{cofibrations}, and the maps in ${\cW}$ are called \emph{weak equivalences}.
\end{define}

\begin{define}\label{d:almost_admiss_dual}
A weak cofibration category $(\cC,\cW,\cC of)$ is called
\begin{enumerate}
  \item ind-admissible, if the class $\Lw^{\cong}(\cW)$, of morphisms in $\Ind(\cC)$, satisfies the two out of three property.
  \item almost ind-admissible, if the class $\Lw^{\cong}(\cW)$, of morphisms in $\Ind(\cC)$, satisfies the following portion of the two out of three property:

For every pair $X\xrightarrow{f} Z\xrightarrow{g} Y $ of composable morphisms in $\Ind(\cC)$ we have:
\begin{enumerate}
\item If $f,g$ belong to $\Lw^{\cong}(\cW)$ then $g\circ f\in \Lw^{\cong}(\cW)$.
\item If $f,g\circ f$ belong to $\Lw^{\cong}(\cW)$ then $g\in \Lw^{\cong}(\cW)$.
\end{enumerate}
\end{enumerate}
\end{define}

\begin{thm}[{\cite[Theorem 3.14]{BaSc2}}]\label{t:almost_model_dual}
Let $(\cC,\mcal{W},\cC of)$ be a small almost ind-admissible weak cofibration category.
Then there exists an almost model category structure on $\Ind(\cC)$ such that:
\begin{enumerate}
\item The weak equivalences are $\mathbf{W} := \Lw^{\cong}(\mcal{W})$.
\item The fibrations are $\mathbf{F} := (\cC of\cap \mcal{W})^{\perp} $.
\item The cofibrations are $\mathbf{C} := \R(\coSp^{\cong}(\cC of))$.
\end{enumerate}
Furthermore, we have $\mathbf{F}  \cap \mathbf{W}=  \cC^{\perp}$ and $\mathbf{C}\cap\mathbf{W} = \R(\coSp^{\cong}(\cC of\cap{\cW})).$
\end{thm}

\begin{rem}\label{t:model}
If, in Theorem \ref{t:almost_model_dual}, the weak cofibration category $(\cC,\cW,\cC of)$ is also ind-admissible, then the almost model structure on $\Ind(\cC)$ described there is clearly a model structure.
\end{rem}

\section{Left proper weak cofibration categories}\label{s:r_proper}

In this section we define the notion of a left proper almost model category and a left proper weak cofibration category. We then show that that a small left proper almost ind-admissible weak cofibration category gives rise to a left proper almost model structure on its ind-category.

\begin{define}\label{d:r_proper}
Let $\cC$ be an almost model category or a weak cofibration category. Then $\cC$ is called left proper if for every pushout square of the form
\[
\xymatrix{A\ar[d]^i\ar[r]^f & B\ar[d]^j\\
C\ar[r] & D}
\]
such that $f$ is a cofibration and $i$ is a weak equivalence, the map $j$ is a weak equivalence.
\end{define}

The proof of the following proposition is based on the proof of \cite[Theorem 4.15]{Isa}:
\begin{prop}\label{p:left proper gen}
Let $(\cC,\cW,\cC of)$ be a left proper weak cofibration category.
Then for every pushout square in $\Ind(\cC)$ of the form
\[
\xymatrix{A\ar[d]^i\ar[r]^f & B\ar[d]^j\\
C\ar[r] & B\coprod_A C}
\]
such that $f$ is in $\Lw^{\cong}(\cC of)$ and $i$ is in $\Lw^{\cong}(\cW)$, the map $j$ is in $\Lw^{\cong}(\cW)$.
\end{prop}

\begin{proof}
There exists a diagram in $\Ind(\cC)$
$$\xymatrix{C''\ar[d]^{\cong} & A''\ar[l]\ar[d]^{\cong} & & \\
             C & A\ar[l]^i\ar[d]^{\cong} \ar[r]_f & B\ar[d]^{\cong}\\
               &  A' \ar[r]& B'}$$
such that the vertical maps are isomorphisms in $\Ind(\cC)$ and such that $A'\to B'$ is a natural transformation indexed by $I$ that is level-wise in $\cC of$ and $A''\to C''$ is a natural transformation indexed by $J$ that is level-wise in $\cW$.

Let $A''\xrightarrow{\cong} A'$ denote the composition $A''\xrightarrow{\cong}A\xrightarrow{\cong} A'$. It is an isomorphism in $\Ind(\cC)$. It follows from \cite{AM} Appendix 3.2 that there exists a cofiltered category $K$, cofinal functors $p:K\to I$ and $q:K\to J$ and a map in $\cC^K$
$$q^*A''\xrightarrow{}p^*A'$$
such that there is a commutative diagram in $\Ind(\cC)$
$$\xymatrix{A''\ar[r]^{\cong} \ar[d]_{\nu_{q,A''}}^{\cong} & A'\ar[d]_{\nu_{p,A'}}^{\cong}\\
             q^*A''\ar[r]^{\cong} & p^*A'}$$
with all maps isomorphisms (see Lemma \ref{l:cofinal}). Thus, we have a diagram in $\cC^K$
$$q^*C''\xleftarrow{i''} q^*A''\xleftarrow{g''}p^*A'\xrightarrow{f''} p^*B',$$
in which $f''$ is a levelwise cofibration, $i''$ is a levelwise weak equivalence and
$g''$ is a pro-isomorphism (but not necessarily levelwise isomorphism). We also have an isomorphism of diagrams in $\Ind(\cC)$
$$\xymatrix{
q^*C'' \ar[d]^{\cong} & q^*A''\ar[l]_{i''} & p^*A' \ar[l]_{g''}\ar[d]^{\cong}\ar[r]^{f''} & p^*B' \ar[d]^{\cong}\\
C & & A \ar[ll]^{i}  \ar[r]_f& B.
}$$
Thus, the above pushout square is isomorphic, as a diagram in $\Ind(\cC)$, to the following (levelwise) pushout:
$$\xymatrix{
 p^*A' \ar[r]^{g''} \ar[d]_{f''} & q^*A''  \ar[rr]^{i''}\ar[d]_{f'''} & & q^*C'' \ar[d]\\
p^*B' \ar[r]^{g'''} &  q^*A'' \coprod_{p^*A'} p^*B'\ar[rr]^{i'''} & &q^*C'' \coprod_{p^*A'} p^*B'.
}$$
It thus remains to show that the map $ p^*B'\to q^*C''\coprod_{p^*A'} p^*B'$ is in $\Lw^{\cong}(\cW)$.

Because $g'''$ is an isomorphism it suffices to show that
$i'''$ is a levelwise weak equivalence.
Since pushouts preserve cofibrations in $\cC$, we know that $f'''$ is a levelwise cofibration.
Now the map $i'''$ is a levelwise pushout of a weak equivalence along a cofibration, so it
is a levelwise weak equivalence because $\cC$ is left proper.
\end{proof}

\begin{cor}\label{c:r_proper}
Let $\cC$ be a small left proper almost ind-admissible weak cofibration category.
Then with the almost model structure defined in Theorem \ref{t:almost_model_dual}, $\Ind(\cC)$ is a
left proper almost model category.
\end{cor}

\begin{proof}
The cofibrations in $\Ind(\cC)$ are given by $\R(\coSp^{\cong}(\cC of))$. Using Propositions \ref{l:ret_lw} and \ref{p:forF_sp_is_lw}, we have that
$$\R(\coSp^{\cong}(\cC of))\subseteq \R(\Lw^{\cong}(\cC of))\subseteq \Lw^{\cong}(\cC of),$$
so the result follows from Proposition \ref{p:left proper gen}.
\end{proof}

\section{Tensored and monoidal structures in ind-categories}\label{s:tensored}

\subsection{Two variables adjunctions in ind-categories}\label{ss:adj}
In this subsection we discuss general two variables adjunctions in ind-categories, and how they are induced from bifunctors on the original categories.

\begin{define}\label{d:hom_map}
Let $\cB$,$\cC$,$\cD$ be categories. An adjunction of two variables from $\cB\times \cC$ to $\cD$ is a quintuple $(\otimes,\Hom_r,\Hom_l,\phi_r,\phi_l)$, where $(-)\otimes (-):\cB\times \cC\to \cD$, $\Hom_r(-,-):\cC^{op}\times\cD\to \cB$, $\Hom_l(-,-):\cB^{op}\times \cD\to \cC$ are bifunctors, and $\phi_r,\phi_l$ are natural isomorphisms
$$\phi_r: \cD(B \otimes C,D)\xrightarrow{\cong} \cB(B,\Hom_r(C,D)),$$
$$\phi_l: \cD(B \otimes C,D)\xrightarrow{\cong} \cC(C,\Hom_l(B,D)).$$
\end{define}

\begin{define}\label{d:prolong}
Let $\cB$,$\cC$,$\cD$ be categories and let $(-)\otimes (-):\cB\times \cC\to \cD$ be a bifunctor.
We have a naturally induced prolongation of $\otimes$ to a bifunctor (which we also denote by $\otimes$)
$$(-)\otimes (-):\Ind(\cB)\times \Ind(\cC)\to \Ind(\cD).$$
If $B=\{B_i\}_{i\in I}$ is an object in $\Ind(\cB)$ and $C=\{C_j\}_{j\in J}$ is an object in $\Ind(\cC)$, then  $B\otimes C$ is the object in $\Ind(\cD)$ given by the diagram
$$\{B_{i}\otimes C_{j}\}_{(i,j)\in I\times J}.$$
\end{define}

\begin{prop}\label{p:adj}
Let $\cB$,$\cC$,$\cD$ be small categories that have finite colimits and let $(-)\otimes (-):\cB\times \cC\to \cD$ be a bifunctor.
Suppose that $\otimes$ commutes with finite colimits in every variable separately. Then the prolongation
$$(-)\otimes (-):\Ind(\cB)\times \Ind(\cC)\to \Ind(\cD).$$
is a part of a two variable adjunction $(\otimes,\Hom_r,\Hom_l)$.
\end{prop}

\begin{proof}
Let $\widetilde{\Ind}(\cB)$ denote the full subcategory of the presheaf category $PS(\cB):=Set^{\cB^{op}}$ spanned by those presheaves that commute with finite limits (that is, that transfer finite colimits in $\cB$ to finite limits in $Set$). Consider the Yoneda embedding $j:\cB\to PS(\cB)$. Extend $j$ to $\Ind(\cB)$ by the universal property of $\Ind(\cB)$ (so that the result commutes with filtered colimits) $j:\Ind(\cB)\to PS(\cB)$. It is a classical fact (see for example \cite{AR}) that the extended $j$ induces an equivalence of categories
$$j:\Ind(\cB)\to \widetilde{\Ind}(\cB).$$

Let $C=\{C_t\}_{t\in T}$ be an object in $\Ind(\cC)$. Suppose $B\cong \colim_d B_d$ is a finite colimit diagram in $\cB$. Then we have
$$B\otimes C=\{B\otimes C_t\}_{t\in T}\cong colim^{\Ind(\cD)}_{t\in T}(B\otimes C_t)\cong$$
$$\cong colim^{\Ind(\cD)}_{t\in T}((colim^{\cB}_d B_d)\otimes C_t)\cong colim^{\Ind(\cD)}_{t\in T}(colim^{\cD}_d (B_d\otimes C_t))\cong $$
$$\cong colim^{\Ind(\cD)}_{t\in T}(colim^{\Ind(\cD)}_d (B_d\otimes C_t))\cong colim^{\Ind(\cD)}_d(colim^{\Ind(\cD)}_{t\in T}(B_d\otimes C_t))\cong$$
$$\cong colim^{\Ind(\cD)}_d \{B_d\otimes C_t\}_{t\in T}\cong colim^{\Ind(\cD)}_d (B_d\otimes C)$$
It follows that the functor $(-)\otimes C:\cB\to \Ind(\cD)$ commutes with finite colimits.
For every object $D$ in $\Ind(\cD)$, we can thus define $\widetilde{\Hom}_r(C,D)$ as an object in $\widetilde{\Ind}(\cB)$ by
$$\widetilde{\Hom}_r(C,D)(B):=\Hom_{\Ind(\cD)}(B\otimes C,D),$$
for $B$ in $\cB$. Clearly this defines a functor
$$\widetilde{\Hom}_r:\Ind(\cC)^{op}\times \Ind(\cD)\to \widetilde{\Ind}(\cB).$$
Composing with an inverse equivalence to $j$ we obtain a functor
$$\Hom_r:\Ind(\cC)^{op}\times \Ind(\cD)\to {\Ind}(\cB).$$

We have isomorphisms, natural in $B=\{B_j\}_{j\in J}\in \Ind(\cB)$, $C\in \Ind(\cC)$ and $D\in \Ind(\cD)$
$$\Hom_{\Ind(\cB)}(B,\Hom_r(C,D))\cong lim_{j\in J}\Hom_{\Ind(\cB)}(B_j,\Hom_r(C,D))\cong$$
$$\cong lim_{j\in J}\Hom_{\widetilde{\Ind}(\cB)}(j(B_j),\widetilde{\Hom}_r(C,D))\cong lim_{j\in J}\widetilde{\Hom}_r(C,D)(B_j)\cong$$
$$\cong lim_{j\in J}\Hom_{\Ind(\cD)}(B_j\otimes C,D)\cong Hom_{\Ind(\cD)}(B\otimes C,D).$$

The functor
${\Hom_l}:\Ind(\cB)^{op}\times \Ind(\cD)\to {\Ind}(\cC)$
is defined similarly.
\end{proof}

\subsection{Tensored and monoidal structures in ind-categories}
In this subsection we turn to the special case of tensored and monoidal structures in ind-categories, and how they are induced from similar structures on the original categories.

\begin{define}\label{d:WC_monoidal}
Let $(\cM,\otimes,I)$ be a monoidal category (see for example \cite[Section A.1.3]{Lur}) with finite colimits. We will say that $\cM$ is weakly closed if $\otimes$ commutes with finite colimits in each variable separately.
\end{define}

\begin{define}\label{d:left_action}
Let $(\cM,\otimes,I)$ be a monoidal category and let $\cC$ be a category. A left action of $\cM$ on $\cC$ is a bifunctor $\otimes:\cM\times \cC\to \cC$ together with coherent natural isomorphisms
$$L\otimes(K\otimes X)\cong (K\otimes L)\otimes X,$$
$$I\otimes X\cong X,$$
for $X$ in $\cC$ and $K,L$ in $\cM$.

If we say that $\cC$ is tensored over $\cM$ we mean that we are given a left action of $\cM$ on $\cC$.
\end{define}

\begin{define}\label{d:WC_tensored}
Let $\cM$ be a monoidal category with finite colimits and let $\cC$ be a category with finite colimits. Let $\otimes$ be a left action of $\cM$ on $\cC$. We will say that this action is weakly closed if $\otimes$ commutes with finite colimits in each variable separately.
\end{define}

The following two lemmas are clear, but we include them for later
reference.
\begin{lem}
Let $(\cM,\otimes,I)$ be a (symmetric) monoidal category. Then the natural prolongation
$$(-)\otimes (-):\Ind(\cM)\times \Ind(\cM)\to \Ind(\cM),$$
makes $(\Ind(\cM),\otimes,I)$ into a (symmetric) monoidal category.
\end{lem}
\begin{lem}
Let $\cM$ be a monoidal category and let $\cC$ be a category. Let $\otimes$ be a left action of $\cM$ on $\cC$. Then the natural prolongation
$$(-)\otimes (-):\Ind(\cM)\times \Ind(\cC)\to \Ind(\cC),$$
is a left action of the monoidal category $\Ind(\cM)$ on $\Ind(\cC)$.
\end{lem}

\begin{prop}\label{l:monoidal}
Let $\cM$ be a small weakly closed (symmetric) monoidal category with finite colimits (see Definition \ref{d:WC_monoidal}). Then the natural prolongation
$$(-)\otimes (-):\Ind(\cM)\times \Ind(\cM)\to \Ind(\cM),$$
makes $\Ind(\cM)$ a closed (symmetric) monoidal category.
\end{prop}

\begin{proof}
By Proposition \ref{p:adj} the prolongation
$$(-)\otimes (-):\Ind(\cM)\times \Ind(\cM)\to \Ind(\cM),$$
is a part of a two variable adjunction $(\otimes,\Hom_r,\Hom_l)$.
\end{proof}

\begin{prop}\label{p:enriched}
Let $\cM$ be a small monoidal category with finite colimits and let $\cC$ be a small category with finite colimits. Let $\otimes$ be a weakly closed left action of $\cM$ on $\cC$ (see Definition \ref{d:WC_tensored}). Then the natural prolongation
$$(-)\otimes (-):\Ind(\cM)\times \Ind(\cC)\to \Ind(\cC),$$
makes $\Ind(\cC)$ enriched tensored and cotensored (see for example \cite[Section A.1.4]{Lur}) over the monoidal category $\Ind(\cM)$.
\end{prop}

\begin{proof}
By Proposition \ref{p:adj} the natural prolongation
$$(-)\otimes (-):\Ind(\cM)\times \Ind(\cC)\to \Ind(\cC),$$
is a part of a two variable adjunction, which we now denote $(\otimes,\Map,\hom)$.
Note that
$$\Map(-,-):\Ind(\cC)^{op}\times \Ind(\cC)\to \Ind(\cM),$$
$$\hom(-,-):\Ind(\cM)^{op}\times \Ind(\cC)\to \Ind(\cC).$$
\end{proof}

\begin{rem}\label{r:prolong cartesian}
Suppose that $\cM$ is a small category with finite limits and colimits. Then $\cM$ is a monoidal category with respect to the categorical product.
By Definition \ref{d:prolong}, the cartesian product $(-)\times (-):\cM\times \cM\to \cM$ has a natural prolongation to a bifunctor, which we denote
$$(-)\otimes (-):\Ind(\cM)\times \Ind(\cM)\to \Ind(\cM).$$
It is not hard to see that $\otimes$ is exactly the categorical product in $\Ind(\cM)$. Let $B=\{B_i\}_{i\in I}$ and $C=\{C_j\}_{j\in J}$ be objects in $\Ind(\cM)$. Then
$B=\colim_{i\in I}B_i$ and $C=\colim_{j\in J}C_j$ in $\Ind(\cM)$, so we obtain
the following natural isomorphisms in $\Ind(\cM)$:
$$B\times C\cong (\colim_{i\in I}B_i)\times (\colim_{j\in J}C_j)\cong\colim_{i\in I}\colim_{j\in J}(B_i\times C_j)$$
$$\cong\colim_{(i,j)\in I\times J}(B_{i}\times C_{j})\cong \{B_{i}\times C_{j}\}_{(i,j)\in I\times J}\cong B\otimes C.$$
Here we have used the fact that the category $\Ind(\cM)$ is finitely locally presentable, so filtered colimits commute with finite limits in $\Ind(\cM)$ (see \cite{AR}) and the fact that by \cite[Corollary 3.19]{BaSc} we know that for every $A,B\in\cM$, the product $A\times B$ is the same in $\cM$ and in $\Ind(\cM)$.
\end{rem}

\section{Tensored and monoidal weak cofibration categories}\label{s:almost}

\subsection{Left quillen bifunctors between almost model categories}\label{s:Quillen}

In this subsection we discuss the notion of a left Quillen bifunctor in the context of almost model categories.
For the notion for usual model categories see for instance \cite[Chapter 4]{Hov}. We also discuss an analogous notion for weak cofibration categories, which we call a \emph{weak left Quillen bifunctor}. We then show that a weak left Quillen bifunctor between small almost ind-admissible weak cofibration categories, gives rise to a left Quillen bifunctor between the corresponding almost model structures on their ind-categories.

\begin{define}\label{d:LQB}
Let $\cB$,$\cC$,$\cD$ be almost model categories and let $(-)\otimes (-):\cB\times \cC\to \cD$ be a bifunctor.
The bifunctor $\otimes$ is called a \emph{left Quillen bifunctor} if $\otimes$ is a part of a two variable adjunction $(\otimes,\Hom_r,\Hom_l)$ (see Definition \ref{d:hom_map}), and for every cofibration $j:X\to Y$ in $\cB$ and every cofibration $i:L\to K$ in $\cC$ the induced map
$$X \otimes K \coprod_{X \otimes L}Y\otimes L\to Y\otimes K$$
is a cofibration (in $\cD$), which is acyclic if either $i$ or $j$ is.
\end{define}

Taking $B=*$ to be the trivial category in the definition above, we get the notion of a left Quillen functor between almost model categories.
\begin{define}\label{d:LQF}
Let $F:\mcal{C}\to \mcal{D}$ be a functor between two almost model categories. Then $F$ is called a \emph{left Quillen functor} if $F$ is a left adjoint and $F$ preserves cofibrations and trivial cofibrations.
\end{define}

The following Proposition can be proven just as in the case of model categories (see for example \cite[Lemma 4.2.2]{Hov}). This is because the proof mainly depends on Lemma \ref{l:lifting}.

\begin{prop}\label{p:Qbifunc}
Let $\cB$,$\cC$,$\cD$ be almost model categories, and let $(\otimes,\Hom_r,\Hom_l)$ be a two variable adjunction.
Then the following conditions are equivalent
\begin{enumerate}
\item The bifunctor $\otimes$ is a left Quillen bifunctor.
\item For every cofibration $j:X\to Y$ in $\cB$ and every fibration $p:A\to B$ in $\cD$ the induced map
$$\Hom_l(Y,A)\to \Hom_l(X,A) \prod_{\Hom_l(X,B)}\Hom_l(Y,B)$$
is a fibration (in $\cC$), which is acyclic if either $j$ or $p$ is.
\item For every cofibration $i:L\to K$ in $\cC$ and every fibration $p:A\to B$ in $\cD$ the induced map
$$\Hom_r(K,A)\to \Hom_r(L,A) \prod_{\Hom_r(L,B)}\Hom_r(K,B)$$
is a fibration (in $\cB$), which is acyclic if either $i$ or $p$ is.
\end{enumerate}
\end{prop}

Taking $\cB=*$ to be the trivial category in Proposition \ref{p:Qbifunc}, we get the following corollary, which is well known for model categories:
\begin{cor}\label{c:Qfunc}
Let $\cC$,$\cD$ be almost model categories, and let
$$F:\cC\adj\cD:G$$
be an adjunction.
Then the following conditions are equivalent:
\begin{enumerate}
\item The functor $F$ is a left Quillen functor.
\item The functor $G$ preserves fibrations and trivial fibrations.
\end{enumerate}
In this case we will say that $G$ is a right Quillen functor and that the adjoint pair $F:\cC\adj\cD:G$ is a Quillen pair.
\end{cor}

In \cite{BaSc1} the notion of a \emph{weak right Quillen functor} is defined. The dual notion is the following:
\begin{define}\label{d:WLQF}
Let $F$ be a functor between two weak cofibration categories. Then $F$ is called a \emph{weak left Quillen functor} if $F$ commutes with finite colimits and preserves cofibrations and trivial cofibrations.
\end{define}

We now generalize this to the notion of a \emph{weak left Quillen bifunctor}.
\begin{define}\label{d:WLQB}
Let $\cB$,$\cC$,$\cD$ be weak cofibration categories and let $(-)\otimes (-):\cB\times \cC\to \cD$ be a bifunctor.

The bifunctor $\otimes$ is called a \emph{weak left Quillen bifunctor} if $\otimes$ commutes with finite colimits in every variable separately, and for every cofibration $j:X\to Y$ in $\cB$ and every cofibration $i:L\to K$ in $\cC$ the induced map
$$X \otimes K \coprod_{X \otimes L}Y\otimes L\to Y\otimes K$$
is a cofibration (in $\cD$), which is acyclic if either $i$ or $j$ is.
\end{define}

\begin{rem}.
\begin{enumerate}
\item If $(-)\otimes (-):\cB\times \cC\to \cD$ is a weak left Quillen bifunctor between \emph{model} categories, then $\otimes$ is not necessarily a left Quillen bifunctor, since $\otimes$ is not assumed to be a part of a two variable adjunction.

\item Let $(-)\otimes (-):\cB\times \cC\to \cD$ be a weak left Quillen bifunctor. Then for every cofibrant object $B$ in $\cB$ and every cofibrant object $C$ in $\cC$, the functors $B\otimes (-):\cC\to\cD$ and $(-)\otimes C:\cB\to\cD$, are weak left Quillen functors (see Definition \ref{d:WLQF}).
\end{enumerate}
\end{rem}

The main fact we want to prove about weak left Quillen bifunctors is the following:

\begin{thm}\label{p:RQFunc}
Let $(-)\otimes (-):\cB\times \cC\to \cD$ be a weak left Quillen bifunctor between small almost ind-admissible weak cofibration categories. Then the prolongation of $\otimes$ (See Definition \ref{d:prolong})
$$(-)\otimes (-):\Ind(\cB)\times \Ind(\cC)\to \Ind(\cD),$$
is a left Quillen bifunctor relative to the almost model structures defined in Theorem \ref{t:almost_model_dual}.
\end{thm}

\begin{proof}
By Proposition \ref{p:adj} the prolongation
$$(-)\otimes (-):\Ind(\cB)\times \Ind(\cC)\to \Ind(\cD).$$
is a part of a two variable adjunction $(\otimes,\Hom_r,\Hom_l)$.
\begin{lem}
Let $j:X\to Y$ be a cofibration in $\cB$ and let $p:A\to B$ be a fibration in $\Ind(\cD)$. Then the induced map
$$\Hom_l(Y,A)\to \Hom_l(X,A) \prod_{\Hom_l(X,B)}\Hom_l(Y,B)$$
is a fibration (in $\Ind(\cC)$), which is acyclic if either $j$ or $p$ is.
\end{lem}

\begin{proof}
Let $i:L\to K$ be an acyclic cofibration in $\cC$ and let
$$\xymatrix{
L\ar[r]\ar[d] & \Hom_l(Y,A) \ar[d]\\
K\ar[r]       & \Hom_l(X,A) \prod_{\Hom_l(X,B)}\Hom_l(Y,B).}
$$
be a commutative square. We want to show that there exists a lift in the above square. By adjointness it is enough to show that there exists a lift in the induced square

$$\xymatrix{
X \otimes K \coprod_{X \otimes L}Y\otimes L\ar[r]\ar[d] & A \ar[d]\\
Y\otimes K\ar[r]      & B.}$$
The left vertical map in the square above is an acyclic cofibration in $\cD$, since $\otimes$ is a weak left Quillen bifunctor, and the right vertical map is a fibration in $\Ind(\cD)$. We thus get, by Theorem \ref{t:almost_model_dual}, that there exists a lift in the square above, and so
$$\Hom_l(Y,A)\to \Hom_l(X,A) \prod_{\Hom_l(X,B)}\Hom_l(Y,B)$$
is a fibration in $\Ind(\cC)$.

Now suppose that either $j$ or $p$ is acyclic.
Let $i:L\to K$ be a cofibration in $\cC$ and let
$$\xymatrix{
L\ar[r]\ar[d] & \Hom_l(Y,A) \ar[d]\\
K\ar[r]       & \Hom_l(X,A) \prod_{\Hom_l(X,B)}\Hom_l(Y,B).}
$$
be a commutative square. We want to show that there exists a lift in the above square. By adjointness it is enough to show that there exists a lift in the induced square
$$\xymatrix{
X \otimes K \coprod_{X \otimes L}Y\otimes L\ar[r]\ar[d] & A \ar[d]\\
Y\otimes K\ar[r]      & B.}$$
If $j$ is acyclic, then the left vertical map in the square above is an acyclic cofibration in $\cD$, and the right vertical map is a fibration in $\Ind(\cD)$. If $p$ is acyclic, then the left vertical map in the square above is a cofibration in $\cD$, and the right vertical map is an acyclic fibration in $\Ind(\cD)$. Anyway we get, by Theorem \ref{t:almost_model_dual}, that there exists a lift in the square above, and so
$$\Hom_l(Y,A)\to \Hom_l(X,A) \prod_{\Hom_l(X,B)}\Hom_l(Y,B)$$
is an acyclic fibration in $\Ind(\cC)$.
\end{proof}

By Proposition \ref{p:Qbifunc}, in order to finish the proof it is enough to show that for every cofibration $i:L\to K$ in $\Ind(\cC)$ and every fibration $p:A\to B$ in $\Ind(\cD)$ the induced map
$$\Hom_r(K,A)\to \Hom_r(L,A) \prod_{\Hom_r(L,B)}\Hom_r(K,B)$$
is a fibration (in $\Ind(\cB)$), which is acyclic if either $i$ or $p$ is.

Let $i:L\to K$ be cofibration in $\Ind(\cC)$ and let $p:A\to B$ be a fibration in $\Ind(\cD)$.

Let $j:X\to Y$ be an acyclic cofibration in $\cB$ and let
$$\xymatrix{
X\ar[r]\ar[d] & \Hom_r(K,A) \ar[d]\\
Y\ar[r]       & \Hom_r(L,A) \prod_{\Hom_r(L,B)}\Hom_r(K,B).}
$$
be a commutative square. We want to show that there exists a lift in the above square. By adjointness it is enough to show that there exists a lift in the induced square
$$\xymatrix{
L\ar[r]\ar[d] & \Hom_l(Y,A) \ar[d]\\
K\ar[r]       & \Hom_l(X,A) \prod_{\Hom_l(X,B)}\Hom_l(Y,B).}
$$
The left vertical map in the square above is a cofibration in $\Ind(\cC)$, and the right vertical map is an acyclic fibration in $\Ind(\cC)$, by the lemma above. We thus get, by Theorem \ref{t:almost_model_dual}, that there exists a lift in the square above, and so
$$\Hom_r(K,A)\to \Hom_r(L,A) \prod_{\Hom_r(L,B)}\Hom_r(K,B)$$
is a fibration in $\Ind(\cB)$.

Now suppose that either $i$ or $p$ is acyclic.
Let $j:X\to Y$ be a cofibration in $\cB$ and let
$$\xymatrix{
X\ar[r]\ar[d] & \Hom_r(K,A) \ar[d]\\
Y\ar[r]       & \Hom_r(L,A) \prod_{\Hom_r(L,B)}\Hom_r(K,B).}
$$
be a commutative square. We want to show that there exists a lift in the above square. By adjointness it is enough to show that there exists a lift in the induced square
$$\xymatrix{
L\ar[r]\ar[d] & \Hom_l(Y,A) \ar[d]\\
K\ar[r]       & \Hom_l(X,A) \prod_{\Hom_l(X,B)}\Hom_l(Y,B).}
$$
If $i$ is acyclic, then the left vertical map in the square above is an acyclic cofibration in $\Ind(\cC)$, and the right vertical map is a fibration in $\Ind(\cC)$ by the lemma above. If $p$ is acyclic, then the left vertical map in the square above is a cofibration in $\Ind(\cC)$, and the right vertical map is an acyclic fibration in $\Ind(\cC)$ by the lemma above. Anyway we get, by Theorem \ref{t:almost_model_dual}, that there exists a lift in the square above, and so
$$\Hom_r(K,A)\to \Hom_r(L,A) \prod_{\Hom_r(L,B)}\Hom_r(K,B)$$
is an acyclic fibration in $\Ind(\cB)$.
\end{proof}

Taking $B=*$ to be the trivial category in Theorem \ref{p:RQFunc}, we get the following corollary, which was shown in \cite{BaSc1} for pro-admissible weak fibration categories:
\begin{cor}
Let $F:\cC\to \cD$ be a weak left Quillen functor between small almost ind-admissible weak cofibration categories. Then the prolongation of $F$
$$F:\Ind(\cC)\to \Ind(\cD),$$
is a left Quillen functor relative to the almost model structures defined in Theorem \ref{t:almost_model_dual}.
\end{cor}

\subsection{Tensored and monoidal almost model categories}
In this subsection we define the notions of tensored and monoidal almost model categories, as direct generalization of the corresponding notions in the world of model categories (see for example \cite[Section A.3.1]{Lur}).

We also define the notions of tensored and monoidal weak cofibration categories, and show that they induce the corresponding notions for the almost model structures on their ind-categories, if they are small and almost ind-admissible.

\begin{define}\label{d:monoidal_almost}
Let $(\cM,\otimes,I)$ be a monoidal category which is also an almost model category. We will say that $\cM$, with this structure, is a monoidal almost model category if the following conditions are satisfied:
\begin{enumerate}
\item $\otimes:\cM\times \cM\to \cM$ be a left Quillen bifunctor (see Definition \ref{d:LQB}).
\item $I$ is a cofibrant object in $\cM$.
\end{enumerate}
We denote the two variable adjunction of $\otimes$ by $(\otimes,\Hom,\Hom)$.
Note that
$$\Hom(-,-):\Ind(\cC)^{op}\times \Ind(\cM)\to \Ind(\cM).$$
\end{define}

\begin{define}\label{d:tensored_almost}
Let $\cM$ be a monoidal almost model category, and let $\cC$ be an almost model category category which is also tensored over $\cM$. We say that $\cC$, with this structure, is an $\cM$-enriched almost model category, if $(-)\otimes (-):\cM\times \cC\to \cC$ is a left Quillen bifunctor (see Definition \ref{d:LQB}).

We denote the two variable adjunction of $\otimes$ by $(\otimes,\Map,\hom)$.
Note that
$$\Map(-,-):\Ind(\cC)^{op}\times \Ind(\cC)\to \Ind(\cM),$$
$$\hom(-,-):\Ind(\cM)^{op}\times \Ind(\cC)\to \Ind(\cC).$$
\end{define}

\begin{define}\label{d:monoidal}
Let $(\cM,\otimes,I)$ be a weak cofibration category which is also a monoidal category. We will say that $\cM$, with this structure, is a monoidal weak cofibration category if the following conditions are satisfied:
\begin{enumerate}
\item $\otimes:\cM\times \cM\to \cM$ be a weak left Quillen bifunctor (see \ref{d:WLQB}).
\item $I$ is a cofibrant object in $\cM$.
\end{enumerate}
\end{define}

\begin{rem}
Let $\cM$ be a monoidal weak cofibration category. Then since $\otimes$ is a weak left Quillen bifunctor, we get in particular that $\cM$ is weakly closed (see Definition \ref{d:WC_monoidal}).
\end{rem}

\begin{prop}\label{p:monoidal}
Let $(\cM,\ten,I)$ be a small almost ind-admissible monoidal weak cofibration category. Then with the almost model structure described in Theorem \ref{t:almost_model_dual} and with the natural prolongation of $\ten$, $\Ind(\cM)$ is a monoidal almost model category (See Definition \ref{d:monoidal_almost}).
\end{prop}

\begin{proof}
By Proposition \ref{l:monoidal}, $\Ind(\cM)$ is naturally a closed monoidal category. The monoidal unit in $\Ind(\cM)$ is $I$ which is cofibrant in $\cM$ and hence also in $\Ind(\cM)$.

It remains to check that
$$(-)\otimes (-):\Ind(\cM)\times \Ind(\cM)\to \Ind(\cM),$$
is a left Quillen bifunctor relative to the almost model structure defined in Theorem \ref{t:almost_model_dual}. But this follows from Theorem \ref{p:RQFunc}.
\end{proof}

\begin{define}\label{d:tensored}
Let $\cM$ be a monoidal weak cofibration category, and let $\cC$ be a weak cofibration category which is also tensored over $\cM$. We say that $\cC$, with this structure, is an $\cM$-enriched weak cofibration category if $(-)\otimes (-):\cM\times \cC\to \cC$ is a weak left Quillen bifunctor (See Definition \ref{d:WLQB}).
\end{define}

\begin{prop}\label{p:tensored}
Let $\cM$ be a small almost ind-admissible monoidal weak cofibration category, and let $\cC$ be a small almost ind-admissible $\cM$-enriched weak cofibration category. Then with the almost model structure described in Theorem \ref{t:almost_model_dual}  and with the natural prolongation of the monoidal product and action, $\Ind(\cC)$ is an $\Ind(\cM)$-enriched almost model category (See Definition \ref{d:tensored_almost}).
\end{prop}

\begin{proof}
By Propositions \ref{p:enriched} and \ref{p:monoidal}, we know that $\Ind(\cM)$ is a monoidal almost model category and that $\Ind(\cC)$ naturally tensored over $\Ind(\cM)$. It thus remains to check that
$$\otimes:\Ind(\cM)\times \Ind(\cC)\to \Ind(\cC),$$
is a left Quillen bifunctor relative to the almost model structures defined in Theorem \ref{t:almost_model_dual}. But this follows from Theorem \ref{p:RQFunc}.
\end{proof}

\subsection{Simplicial almost model categories}\label{s:simplicial}
In this subsection we define the notion of a simplicial almost model category, as direct generalization of the notion of a simplicial model category (see for example \cite[Section A.3.1]{Lur}).
We also define the notion of a simplicial weak cofibration category, and show that it induces the notion of a simplicial almost model category on its ind-category, if it is small and almost ind-admissible.

Let $\cS_f$ denote the category of finite simplicial sets, that is, $\cS_f$ is the full subcategory of finitely presentable objects in the category of simplicial sets $\cS$. There is a natural equivalenve of categories $\Ind(\cS_f)\xrightarrow{\sim}\cS$, given by taking colimits (see \cite{AR}).

\begin{define}\label{d:S_f_cofib}
A map $X\to Y$ in $\cS_f$ will be called:
\begin{enumerate}
\item A cofibration, if it is one to one (at every degree).
\item A weak equivalence, if the induced map of geometric realizations: $|X|\to |Y|$ is a weak equivalence of topological spaces.
\end{enumerate}
\end{define}

The following Proposition is shown in \cite{BaSc2}:
\begin{prop}\label{p:S_f_cofib}
With the weak equivalences and cofibrations given in Definition \ref{d:S_f_cofib}, $\cS_f$ becomes an essentially small finitely complete ind-admissible weak cofibration category.
\end{prop}

It is explained in \cite{BaSc2} that, under the natural equivalence $\cS\cong \Ind(\cS_f)$, the model structure on $\Ind(\cS_f)$ given by Theorem \ref{t:almost_model_dual} is the standard model structure on $\cS$.

\begin{lem}\label{l:S_f_monoidal}
Under the categorical product, $\cS_f$ is a monoidal weak cofibration category.
\end{lem}

\begin{proof}
The product unit $\Delta^0$ is cofibrant, like every object in $\cS_f$.
The categorical product clearly commutes with finite colimits in every variable. It thus remains to check the pushout product axiom for the categorical product in $\cS_f$. But this is a classical result, see for example \cite[Theorem 3.3.2]{Hov}. (Note, however, that we only need to verify the condition for \emph{finite} simplicial sets.)
\end{proof}

We obtain the following result which is well known:

\begin{cor}
With the model structure described in Theorem \ref{t:almost_model_dual} (which is the standard one), the category of simplicial sets $\cS\simeq \Ind(\cS_f)$ is a cartesian monoidal model category.
\end{cor}

\begin{proof}
We have shown in  Remark \ref{r:prolong cartesian} that the
natural prolongation of the categorical product in $\cS_f$ to a bifunctor
$$\Ind(\cS_f)\times \Ind(\cS_f)\to \Ind(\cS_f)$$
is exactly the categorical product in $\Ind(\cS_f)$.
Now the corollary follows from Lemma \ref{l:S_f_monoidal} and Proposition \ref{p:monoidal}.
\end{proof}

We are now able to give the following:
\begin{define}\label{d:simplicial}\
\begin{enumerate}
\item A simplicial weak cofibration category is just an $\cS_f$-enriched weak cofibration category (see Definition \ref{d:tensored}).
\item A simplicial almost model category is just an $\cS$-enriched almost model category (see Definition \ref{d:tensored_almost}).
\end{enumerate}
\end{define}

\begin{prop}\label{p:simplicial}
Let $\cC$ be a small almost ind-admissible simplicial weak cofibration category. Then with the almost model structure described in Theorem \ref{t:almost_model_dual}, $\Ind(\cC)$ is a simplicial almost model category.
\end{prop}

\begin{proof}
This follows directly from Proposition \ref{p:tensored}.
\end{proof}


Department of Mathematics, Hebrew University of Jerusalem, Jerusalem, Israel.

\emph{E-mail address}:
\texttt{ilanbarnea770@gmail.com}

\end{document}